\documentclass[12pt]{amsart}
\usepackage[english]{babel}
\usepackage[latin1]{inputenc}
\usepackage{amssymb,amsfonts,amsmath,amsthm}
\usepackage{graphicx}
\usepackage[all]{xy}
\usepackage{enumerate}
\newcommand{\SortNoop}[1]{}

\newcommand{\pt}{\forall}

\newcommand{\plonge}{\hookrightarrow}

\newcommand{\projects}{\twoheadrightarrow}

\newcommand{\laction}{\curvearrowright}
\newcommand{\ol}[1]{\overline{#1}}

\newcommand{\mc}[1]{\mathcal{#1}}
\newcommand{\lie}[1]{\mathfrak{#1}}

\newcommand{\Ad}{\mathrm{Ad}}

\newcommand{\moduli}[1]{\mathrm{Higgs}(#1)}

\newcommand{\Ker}{\textrm{Ker}}
\newcommand{\Isom}{\textrm{Isom}}

\newcommand{\benum}{\begin{enumerate}}
\newcommand{\eenum}{\end{enumerate}}

\newcommand{\R}{\mathbb{R}}

\newcommand{\C}{\mathbb{C}}
\newcommand{\Z}{\mathbb{Z}}

\renewcommand{\sl}{\mathfrak{sl}}

\newcommand{\su}{\mathfrak{su}}

\newcommand{\so}{\mathfrak{so}}

\newcommand{\Hom}{\mathrm{Hom}}
\newcommand{\End}{\mathrm{End}}
\newcommand{\Pic}{\mathrm{Pic}}
\newcommand{\Jac}{\mathrm{Jac}}

\newcommand{\SL}{\mathrm{SL}}

\newcommand{\SU}{\mathrm{SU}}
\newcommand{\U}{\mathrm{U}}
\newcommand{\GL}{\mathrm{GL}}

\newcommand{\g}{\mathfrak{g}}
\newcommand{\h}{\mathfrak{h}}

\newcommand{\la}{\mathfrak{a}}
\newcommand{\lr}{\mathfrak{r}}

\newcommand{\m}{\mathfrak{m}}

\newcommand{\lt}{\mathfrak{t}}

\newcommand{\lc}{\mathfrak{c}}
\newcommand{\ld}{\mathfrak{d}}

\newcommand{\gr}{\g_{reg}}
\newcommand{\hgr}{\widehat{\g}_{reg}}

\newcommand{\mr}{{\m}_{reg}}

\newcommand{\CH}{C_H(\la)}
\newcommand{\NH}{N_H(\la)}

\newcommand{\Higgs}{\mathrm{Higgs}}

\newcommand{\spec}{\widetilde{X}}

\newcommand{\specc}{\widehat{X}}

\newcommand{\GN}{\ol{G^\C/N}}
\newcommand{\GD}{\ol{G^\C/D^\C}}
\newcommand{\hbase}{\mc{A}}

\newtheorem{thm}{Theorem}[section]
\newtheorem{prop}[thm]{Proposition}
\newtheorem{cor}[thm]{Corollary}
\newtheorem{rk}[thm]{Remark}
\newtheorem{lm}[thm]{Lemma}
\theoremstyle{definition}

\newtheorem{defi}[thm]{Definition}
\newtheorem{ex}[thm]{Example}

\begin{document}
%%%%%%%%%%%%%%%%%%%%%%%%%%%%%%%%%%%%%%%%%%%%%%%%%%%%%%%%%%%%%%%%%%%
\title[Cameral data for $\SU(p+1,p)$-Higgs bundles]
{Cameral data for $\SU(p+1,p)$-Higgs bundles}
\author{Ana Pe{\'on}-Nieto}
\address{Mathematisches Institut, Ruprecht-Karls Universit{\"a}t\\
Im Neuenheimer Feld 288\\
69120 Heidelberg}
\email{apeonnieto@mathi.uni-heidelberg.de}
\thanks{This work was partially funded by the Mathematics Center Heidelberg}
\begin{abstract}
 We study the cameral and spectral data for the moduli space of polystable $\SU(p+1,p)$-Higgs bundles and deduce the latter from
 the former. 
 As an application, we obtain that the Toledo invariant classifies the connected components 
 of the regular fibers of the Hitchin map. 
 \end{abstract}

\maketitle
% \tableofcontents
% \newpage

 %%%%%%%%%%%%%%%%%%%%%%%%%%%%%%%%%%%%%%%%%%%%%%%%%%%%%%%%%%%%%%%%%%%

%%%%%%%%%%%%%%%%%%%%%%%%%%%%%%%%%%%%%%%%%%%%%%%%%%%%%%%%%%%%%%%%%%%%%%%%%%%%%%%%%%%%
%%%%%%%%%%%%%%%%%%%%%%%%%%%%    Intro            %%%%%%%%%%%%%%%%%%%%%%%%%%%%
%%%%%%%%%%%%%%%%%%%%%%%%%%%%                            %%%%%%%%%%%%%%%%%%%%%%%%%%%%
%%%%%%%%%%%%%%%%%%%%%%%%%%%%%%%%%%%%%%%%%%%%%%%%%%%%%%%%%%%%%%%%%%%%%%%%%%%%%%%%%%%%
\section{Introduction}
Higgs bundle theory has experienced an enormous
development since its origins in \cite{SDE}, due to the rich geometry of these objects. 
% These objects first appear in Hitchin's work 
% \cite{SDE,Teich}. The development of a theory of moduli spaces independent of the complexification 
% $G^\C$ is due to many authors \cite{MartaOscar,BGGUpq,BGMHitchinKoba,GGMSpR,GciaOliveiraUstar}, the general
% theory being proved in  \cite{GGMHitchinKobayashi}. 

An instance of this is the so called non-abelian Hodge correspondence, which given 
a reductive Lie group $G$, establishes a homeomorphism between the 
moduli space of $G$-Higgs bundles on a Riemann surface $X$, and the moduli space of representations
$\rho:\pi_1(X)\to G$ \cite{SDE, Donaldson, Corlette, SimpsonHodge, GGMHitchinKobayashi}. A $G$-Higgs bundle is thus
naturally seen to be a pair $(E,\phi)$, where $E\to X$ is a 
holomorphic principal $H^\C$-bundle, for $H\leq G$ maximally compact, and
$\phi\in H^0(X, E(\m^\C)\otimes K)$ is the Higgs field. In the former, $K\to X$ denotes the canonical bundle,
$\m$ is the non compact part of the Cartan decomposition
 $ \g=\h\oplus\m$ with complexification $\m^\C$ and $E(\m^\C)$ is the associated bundle via the isotropy representation
 (\ref{eq isotropy}).
 
 The moduli space of polystable $G$-Higgs bundles $\moduli{G}$ is equipped with extra structure determined by the 
\textbf{Hitchin map}
\begin{equation}\label{eq Hitchin intro}
h_G:\moduli{G}\to B_G=\oplus_{i=1}^{rk_\R G}H^0(X,K^{d_i})
\end{equation}
which sends a pair $(E,\phi)$ to the characteristic coefficients of $\phi$; the numbers $d_i-1$ are the 
exponents of $G$, and the space $B_G$ is called the \textbf{Hitchin base}. 

The study of the Hitchin map is essential for understanding the geometry of $\moduli{G}$. In the case of 
complex groups it defines an algebraically completely integrable system \cite{Duke, D95}, 
whose description has been applied to many interesting problems
%, such as the computation of
% some cohomology groups of the moduli space of principal $\SL(2,\C)$ bundles 
\cite{Duke, BNR, KouPantevAut, HT, DPLanglands}. 
% the
% definition of theta line bundles on the moduli space of stable principal 
% $\SL(n,\C)$ bundles 
% \cite{BNR}, 
% the determination of the automorphism group of the moduli space semistable vector bundles 
% \cite{KouPantevAut}, 
% or the description of mirror symmetry for Hitchin systems 
% \cite{HT, DPLanglands}.

For real groups, (\ref{eq Hitchin intro}) is not yet well understood. 
The tool to undertake this problem in full generality is Donagi's cameral construction \cite{D93,D95}, 
and Donagi-Gaitsgory's \cite{DG}, which was adapted to real groups 
in the author's thesis \cite{Tesis} (see also \cite{HKR, Abelian}). 
 When $G$ is a matrix group, Hitchin's spectral techniques \cite{Duke} can also be used to understand the fibers 
of the Hitchin map. This is the approach found in \cite{Emily,FGN,HitSchap,SchapUpp,LauThesis},
all of which deal with some classical groups. 
 Just as in the case of complex groups, the Hitchin map for real groups has been applied to a range of 
 problems, such as the description of a generalization of Teichm\"uller space \cite{Teich} and a test of
Kapustin-Witten's approach to mirror symmetry  \cite{HitLanglands}.

In the present paper, we study the case $G=\SU(p+1,p)$ from both the cameral and the spectral
construction, and illustrate how to deduce the latter from the former. 
The choice of $\SU(p+1,p)$ is not arbitrary, but is due to the particular characteristics
 of the groups $\SU(p,q)$ as $q$ varies. Indeed, both $\SU(p,p)$ and $\SU(p+1,p)$ are quasi-split real forms, 
%  meaning that they contain a subgroup $B$ whose complexification $B^\C$ is a Borel subgroup of the 
%  respective complexifications $\SL(2p,\C)$ or $\SL(2p+1,\C)$.
%  This 
(which implies in particular that the fibers of the Hitchin map (\ref{eq Hitchin intro}) 
are generically subvarieties of line bundles only in these two cases \cite{Tesis}), but 
quite different from one another 
($\SU(p,p)$ being of Hermitian tube type and $\SU(p+1,p)$ of Hermitian non-tube type). With respect
to the Hitchin map, this causes for the fibers of the Hitchin map to be generically regular and stable in the
first case and never completely so in the second. 

The case $p=q$ is studied  in \cite{SchapUpp} from the spectral data point of view. The cameral and spectral
 approaches being equivalent for classical groups \cite{D93}, we explain in here how to recover 
 the spectral picture from the cameral one for $\SU(p+1,p)$, the case of $\SU(p,p)$ following in a similar
 way. %We also include a study of the reductive group $\U(p+1,p)$. 

We describe the generic fibers of the Hitchin map (\ref{eq Hitchin intro}) in two ways. Firstly, in terms of principal
$(\C^\times)^{2p}$ bundles $P$ (called cameral data) over the cameral 
cover $\specc\to X$. This is a ramified $S_{2p+1}$-Galois cover of $X$ parametrizing ordered
eigenvalues of Higgs fields. Hence, to each 
Higgs bundle $(E,\phi)$ one can associate a cameral cover, which only depends on the characteristic
polynomial of $\phi$, or equivalently, on the image of $(E,\phi)$ via (\ref{eq Hitchin intro}). Using some
equivariance properties  of cameral data, one can prove it to be determined by $(\C^\times)^p$-principal
bundles over a $\Z_2$-quotient of $\specc$, where the action of $\Z_2$ is determined by the involution defining
$\SU(p+1,p)$ inside of $\SL(2p+1,\C)$. See Theorem \ref{prop gric cam data} for details. 

Secondly, we show in Theorem \ref{thm spectral data}
% is the intersection
% of the torus $(\C^\times)^{2p}$ with $S(\GL(p+1,\C)\times \GL(p,\C))$.
% The difference
% with the complex group case is equivariance of cameral data with respect to an involution
% $\theta$ on the cameral cover $\specc$, induced by the Cartan involution on $\su(p+1,p)$, which 
% allows to establish an isomorphism between the Hitchin fibers and a variety of principal 
% $(\C^\times)^p$-bundles on the quotient $\specc/\theta$. Here  $(\C^\times)^p$ is the intersection
% of the torus $(\C^\times)^{2p}$ with $S(\GL(p+1,\C)\times \GL(p,\C))$.
that the fibers are isomorphic to varieties of line bundles over the spectral cover
$\spec\to X$, which parametrizes (unordered) eigenvalues, and is thus a quotient of the cameral cover 
\cite{D93,DG}. In these terms, the points of the fiber are line bundles over a quotient
$\spec/\Z_2$, plus some extra data. 

The equivalence of both the cameral and spectral approach is proved in Proposition \ref{prop descent}.

The paper is structured as follows: Sections \ref{section cameral data} and \ref{section higgs bundles} establish the basic 
notions and results from \cite{Tesis, Abelian}, which we use to obtain the description of 
the generic regular fibers in Section \ref{section cam data} in terms of cameral data. We compute the spectral data in 
Section \ref{section spectral}, from which we see in Corollary \ref{cor con comps} that there is only one connected component
of the fibers per invariant. Next, we explain in Section \ref{section from cam to spec} how the spectral picture can be deduced
from the cameral one. Section \ref{section non reg} contains a brief discussion on how to complete the results
to non regular bundles; we observe the existence of a bundle of toric varieties surjecting onto a full dimensional subset of the Hitchin fiber. 
A geometric discussion of the algebraic notion of regularity is included in 
Section \ref{section regularity}.

We were notified that  Baraglia and Schaposnik have obtained related results.

\subsection*{Acknowledgements} The author wishes to thank \'Oscar Garc\'ia-Prada, Christian Pauly and Anna Wienhard for useful comments on a first draft of this paper.
%%%%%%%%%%%%%%%%%%%%%%%%%%%%%%%%%%%%%%%%%%%%%%%%%%%%%%%%%%%%%%%%%%%%%%%%%%%%%%%%%%%%
%%%%%%%%%%%%%%%%%%%%%%%%%%%% Quasi-Split and cameral  %%%%%%%%%%%%%%%%%%%%%%%%%%%%
%%%%%%%%%%%%%%%%%%%%%%%%%%%%                            %%%%%%%%%%%%%%%%%%%%%%%%%%%%
%%%%%%%%%%%%%%%%%%%%%%%%%%%%%%%%%%%%%%%%%%%%%%%%%%%%%%%%%%%%%%%%%%%%%%%%%%%%%%%%%%%%
\section{Quasi-split real forms and cameral data}\label{section cameral data}
We briefly explain in this section the main theorem in \cite{Tesis, Abelian} that we will apply to our particular case $\SU(p+1,p)$. 
\begin{defi}
 A quasi-split real form $G<G^\C$ is a real form containing a subgroup $B<G$ whose complexification $B^\C<G^\C$ is a Borel
subgroup. 
\end{defi}
\begin{rk}\label{rk qs}
An alternative definition is the following: a real form is quasi-split if and only if regular elements (cf. Definition 
\ref{defi regular}) have abelian centralisers. The consequence of this is that automorphism preserving the characteristic 
polynomial have abelian connected component.

Quasi-split real forms include split real forms, and in the simple group case, groups whose Lie algebras are $\su(p,p)$, $\su(p+1,p)$, $\so(p,p+2)$ or $\lie{e}_{6(2)}$. 
\end{rk}

We will assume that the involution $\sigma$ defining $G$ inside $G^\C$ commutes with a compact involution $\tau$ in $G^\C$. 
In that case, we can always take $\tau$ such that $G^\tau=H$. In particular, the Cartan involution $\theta$, extended by complex linearity
to $\g^\C$, lifts to a holomorphic involution on $G^\C$ given by $\theta:=\sigma\tau$. As a consequence $H^\C=(G^\C)^\theta$.
This always holds in the connected group 
case (see \cite{K}, Proposition 7.21, or Proposition 3.20 \cite{HKR} for milder conditions on the group). Let $\ld^\C\subset\g^\C$ be a $\theta$ and $\sigma$ 
invariant Cartan subalgebra. We may choose 
it in such a way that $\la^\C=\ld^\C\cap\m^\C$ is maximal. Fix $D^\C=A^\C T^\C$ the corresponding maximal torus, and let
$D^\C<B^\C$ be a Borel subgroup obtained from $B<G$. We let $S\subset\Delta(\g^\C,\ld^\C)$ be the corresponding sets of simple roots
and roots. The roots can be proven to belong to $\la^*\oplus i\lt^*$ (suitably extended by complex
linearity), and so we may choose $S$ in such a way that  
\begin{equation}
\label{condition order}
\la^*> i\lt^*. 
\end{equation}
Quasi-splitness is equivalent to $\Delta\cap i\lt^*=\{0\}$
\begin{defi}
Let $X$ be a connected smooth complex projective curve, and $K$ its canonical bundle.
 A $G$-Higgs bundle on $X$ is a pair  $(E,\phi)$ where $E\to X$ is a holomorphic principal $H^\C$-bundle, and 
 $\phi\in H^0(X, E(\m^\C)\otimes K)$, where $E(\m^\C)$ is the associated bundle via the isotropy representation 
 $\iota: H^\C\to\GL(\m^\C)$.
\end{defi}
We have two meaningful Hitchin maps. To construct them, consider the Chevalley morphisms 
\begin{equation}\label{eq chevalley real}
\chi_G:\m^\C\to\m^\C//H^\C\cong \la^\C/W( \la) 
\end{equation}
(where $W( \la)$ is the restricted Weyl group (\ref{eq res weyl}) and the second isomorphism follows from 
Theorem \ref{thm Chevalley}) and
\begin{equation}\label{eq chevalley complex}
\chi_{G^\C}:\g^\C\to\g^\C//G^\C\cong \ld^\C/W,
\end{equation}
where the second isomorphism is again Theorem \ref{thm Chevalley} applied to the real form $(\g^\C)_\R\subset\g^\C\otimes_\R\C$.
These induce 
\begin{equation}\label{eq rhitchin gral}
 h_G:\Higgs(G)\to B_G=H^0(\la^\C\otimes K/W(\la)), \qquad (E,\phi)\mapsto\chi_G(\phi)
\end{equation}
and
\begin{equation}\label{eq chitchin gral}
h_{G^\C}:\Higgs(G^\C)\to B_{G^\C}=H^0(\ld^\C\otimes K/W),\qquad (E,\phi)\mapsto\chi_{G^\C}(\phi).
\end{equation}
Note that evaluation of (\ref{eq chevalley real}) and (\ref{eq chevalley complex}) on the Higgs field is well defined by 
$H^\C\times\C^\times$-equivariance of $\chi_G$ (respectively, $G^\C\times\C^\times$-equivariance of $\chi_{G^\C}$).

Let $\kappa:\Higgs(G)\to\Higgs(G^\C)$ be given by $\kappa(E,\phi)=(E(G^\C),di(\phi))$, where $i:G\plonge G^\C$ 
is the inclusion. This map induces a commutative diagram
\begin{equation}\label{eq cartesian higgs}
\xymatrix{
\Higgs(G)\ar[r]^-\kappa\ar[d]&\Higgs(G^\C)\ar[d]\\
B_G\ar[r]&B_{G^\C}.
}
\end{equation}
In particular, fibers of $h_G$ are taken to fibers of $h_{G^\C}$.

\begin{defi}\label{defi cam covers}
Given $b\in B_{G^\C}$, we define the associated cameral cover $\specc_b$ to be
the fibered product fitting in the Cartesian diagram 
\begin{equation}\label{eq cartesian cameral}
 \xymatrix{
\specc_b\ar[r]\ar[d]& \ld^\C\otimes K\ar[d]\\
X\ar[r]_-b& \ld^\C\otimes K/W.}
\end{equation}
\end{defi}
Now, assume $b:X\to\ld^\C\otimes K/W$ splits through $\la^\C\otimes K/W(\la)$; that is, $b\in B_G$. 
Then we have a subcover
\begin{equation}\label{eq real cam}
\xymatrix{
\specc_G\ar[r]\ar[d]& \la^\C\otimes K\ar[d]\\
X\ar[r]&  \lr_K,}
\end{equation}
where $\lr_K$ is the image of $\la^\C\otimes K/W(\la)$ in $\ld^\C\otimes K/W$.
% \begin{rk}
% One may evaluate the Hitchin map on any Higgs bundle $(E,\phi)$, and so to any such we may associate a cameral cover.
% \end{rk}
A $G$-Higgs bundle $(E,\phi)$ is said to be \textbf{regular} if for all $x\in X$ we have that $\phi(x)\in\mr$  is a regular 
element of $\m^\C$ (see Definition \ref{defi regular}). 

The following theorem follows from Theorem 4.14 in \cite{Tesis} (see also \cite{Abelian}):
\begin{thm}\label{thm cameral global}
There is a one to one correspondence between isomorphism classes of 

1. Regular $G^\C$-Higgs bundles $(E,\phi)$ (not necessarily polystable) such that $h_{G^\C}(E,\phi)=b$.

2. Principal $D^\C$ bundles $P\to\specc_b$ satisfying 
\begin{itemize}
 \item[CD1] For all $\alpha\in S$, 
 $\gamma_{\alpha}:s_\alpha^*P^{s_\alpha}\otimes R_\alpha\cong P$,
 where $s_\alpha\in W$ is the reflection with respect to $\alpha$, whose action on the principal bundle is defined by
 $P^s=P\times_sD^\C$.  As for $R_\alpha$, it is the principal $D^\C$ bundle obtained by pulling back the divisor 
 $D_\alpha=\{s_\alpha=0\}\subset \ld^\C\otimes K$ to $X$ (thus getting $D_\alpha^X$, see Section 5. in \cite{DG}) and taking 
 $R_\alpha=\check{\alpha}(\mc{O}(D_\alpha^X))$. 
 \item[CD2] $(-\theta)^*P^{w_0}|_{\specc_G}\otimes R_{w_0}\cong P$. Here $w_0\in W$ is the element operating on $\ld^\C$ as $\theta$, 
 and $-\theta$ is the involution on $\specc_b$ induced by the action $-\theta\laction\ld^\C$. The ramification
 $R_{w}=\otimes_{\alpha\in \Delta^+\cap w^{-1}\Delta^-} R_{\alpha}$. By definition
 $R_{w_0}=\otimes_{\alpha\in\Delta_{r}}E_\alpha$, where $\Delta_r$ denotes the real roots, namely, roots 
 $\lambda\in\Delta$ which are obtained from elements in $\la^*$ by extension by complex linearity.
 \end{itemize}
All isomorphisms should be elements of $N$, the normaliser of $D^\C$ in $G^\C$.
\end{thm}
In particular, this theorem allows to study $\moduli{G}$ by means of much simpler moduli spaces of principal $D^\C$-bundles.
\begin{rk}\label{rk cocyclic rels}
 The $R_w$'s define a cocycle $W\to \specc\times BD^\C$, that is, an assignation of a principal $D^\C$-bundle to each $w\in W$ satisfying
 $$
 R_{ww'}\cong w'\cdot R_{w}\otimes R_{w'}
 $$
 canonically. This allows to extend the equivariance properties specified in CD1 of Theorem \ref{thm cameral global} to all elements of the Weyl group and not just reflections
 associated to simple roots.
 \end{rk}
\begin{defi}
A principal $D^\C$-bundle as specified in point \textit{2.} of Theorem \ref{thm cameral global} is called a \textbf{cameral datum}. 
\end{defi}
For the convenience of the reader, we include a discussion of the main elements in the proof of  Theorem \ref{thm cameral global}, which  is based on the study of the gerbe the Hitchin map defines on the level
of the moduli stack of regular Higgs bundles $\mathbf{Higgs}(G)\to\la^\C\otimes K/W(\la)$. When the real 
form is quasi split, this gerbe has abelian band $\mathcal{C}$, that is, locally 
$\mathbf{Higgs}(G)(X)\cong \la^\C\otimes K/W(\la)\times B\mathcal{C}$, a $B\mc{C}$-torsor over $\la^\C\otimes K/W(\la)$. In particular, pullback allows to identify the fibers 
to varieties of 
coherent sheaves of groups on $X$. In order to give the cocyclic description of Proposition 
\ref{prop gric cam data}, one checks that the stack of cameral data is also a $B\mc{C}$-torsor admitting a morphism 
from $\mathbf{Higgs}(G)$, so they are isomorphic. The latter can be done using  \cite{DG} and identifying the right conditions in order for 
cameral data to be induced from a $G$-Higgs bundle. 

To construct a cameral datum from a $G^\C$-Higgs bundle $(E,\phi)$, we reinterpret $\phi$ as a $G^\C\times \C^\times$ equivariant map
$$
\phi:E\times K\to \gr,
$$
where $\gr$ denotes the subset of regular elements of $\g$ and the action of $(g,z)\in G^\C\times\C^\times$ on $x\in\gr$ is given by
$(g,z)\cdot x=z\Ad(g)x$. Now, the Grothendieck-Springer resolution of $\gr$ can be obtained by pullback of two $W$-covers:
\begin{equation}\label{eq GS resolution}
 \xymatrix{
 \GD\ar[d]&\hgr\ar[r]\ar[l]\ar[d]&{\ld^\C}\ar[d]\\
 \GN&\gr\ar[l]\ar[r]_{\chi_{G^\C}}&{\ld^\C}/W.
 }
\end{equation}
Here, $\GN$ is the variety of regular centralisers (a partial compactification of $G^\C/N$ inside of $Gr(r,\g^\C)$, see \cite{DG}), and $\GD$ is the incidence variety of 
$\GN\times G^\C/B^\C$, with $B^\C$ a Borel subgroup containing $D^\C$. 

Note that $b:=\chi_{G^\C}\circ\phi$ is a point of the Hitchin base $B_{G^\C}$. Let $\widehat{E\times K}:=b^*\hgr$. This is easily seen to descend to a cameral cover $\specc\to X$, so that we have a cartesian 
diagram:
$$
\xymatrix{
\widehat{E\times K}\ar[d]\ar[r]&\specc\ar[d]\\
{E\times K}\ar[r]&X.
}$$
As for the cameral datum, $\GD$ is equipped with a universal principal $D^\C$-bundle
$\mc{D}\to\GD$, the pullback of $G^\C/U^\C\to G^\C/B^\C$ via the map $\GD\plonge\GN\times G^\C/B^\C\projects G^\C/B^\C$. The same descent arguments yield a principal bundle
on  $X$, which can be checked to satisfy invariance conditions with respect to the action of the Weyl group. See \cite{DG,NgoLemme} for details on this. 

The last step is to establish the extra conditions for cameral data coming from $b\in H^0(X,\la^\C\otimes K/W(\la))$, which is given by condition CD2 in Theorem \ref{thm cameral global}.
\section{$\SU(p+1,p)$-Higgs bundles and the Hitchin map}\label{section higgs bundles}
For the Lie theory required in this section we refer to Appendix \ref{section appendix}.
\begin{defi}
An $\SU(p+1,p)$-Higgs bundle is a pair $(E=V\oplus W,\phi)$ consisting of a rank $p+1$ vector bundle 
$V$ and a rank $p$ vector bundle $W$ such that $\det (V\oplus W)=\mc{O}_X$, and a Higgs 
field $\phi\in H^0(X,\End(E)\otimes K)$ with
\begin{equation}\label{eq HF}
\phi=\left(
\begin{array}{ccc}
 0&\beta\\
 \gamma&0
 \end{array}
\right), 
\end{equation}
where $\beta \in H^0(X,W^*\otimes V\otimes K)$ and $\gamma\in H^0(X,V^*\otimes W\otimes K)$.
\end{defi}
\begin{defi}
 An $\mbox{SU}(p+1,p)$-Higgs bundle is semistable if for every pair of subbundles $V'\subseteq V$, 
 $W'\subseteq$ 
 such that 
$V'\oplus W'\subsetneq V\oplus W$ and
 $\phi: V'\oplus W'\to (V'\oplus W')\otimes K$, it holds that $\deg V'\oplus W'\leq0$. It is stable
 if it is semistable and the inequality is strict. It is said to be polystable if it is semistable 
 and it decomposes as a direct sum of stable $\U(p_i,q_i)$-Higgs bundles of degree $0$ for suitable
 integers $p_i$ and $q_i$.
\end{defi}
Let $\moduli{\SU(p+1,p)}$ be  the moduli space of $\SU(p+1,p)$-Higgs bundles,  
which is defined to be the
space of 
isomorphism classes of polystable $\SU(p+1,p)$-Higgs bundles \cite{BGGUpq}.

The degree $\deg W$ is a topological invariant of $\moduli{\SU(p+1,p)}$ which hence identifies connected components of the moduli space. 
By Theorem 6.1 in \cite{BGGUpq} the following  \textbf{Milnor-Wood inequality} is satisfied by polystable Higgs bundles:
\begin{equation}\label{eq MW}
0\leq|\deg(W)|\leq p(g-1).
\end{equation}
\begin{rk}
 This invariant is half the Toledo invariant defined in \cite{BGGUpq}, as in the
 fixed determinant case $\deg (W)=-\deg V$. 
\end{rk}
Choose a maximally anisotropic Cartan subalgebra  $(\la')^\C\cong\C^p$ (cf. Defin \ref{defi max aniso}), and let
$(\ld')^\C\cong\C^{2p}$ be a $\theta'$-invariant Cartan subalgebra containing $(\la')^\C$, where $\theta'$ is defined as in (\ref{eq theta 2}). By Lemma \ref{lm invars}, 
imply that \textbf{Hitchin map} (\ref{eq rhitchin gral}) specifies to 
\begin{equation}\label{eq hit moduli}
h:\moduli{\SU(p+1,p)}\to H^0(X,\oplus_{k=1}^{p}K^{2k}) 
\end{equation}
which maps each pair $(E,\phi)$ to the characteristic coefficients of $\phi$, $(a_2,\dots, a_{2p})$. 
More specifically,
$$
a_{2i}=tr(\wedge^{2i}\phi)=
2^itr\left(\wedge^i(\beta\wedge\gamma)\right).
$$
On the other hand, the \textbf{complex Hitchin map }(\ref{eq chitchin gral}) 
reads
$$
h_\C:\moduli{\SL(2p+1,\C)}\to H^0(X,\oplus_{i=1}^{2p}K^i), \ (E,\phi)\mapsto (tr\wedge^i\phi)_{i=1}^{2p}.
$$
Note that $tr\wedge^i\phi=0$ if $2\not|\ i$ for $\phi$ as in (\ref{eq HF}), so we have indeed that $\kappa$ commutes with the respective Hitchin maps.  In what follows, we describe the fibers
of the restriction 
of $h_\C$ to $\kappa(\moduli{\SU(p+1,p)})$ by means of cameral techniques, then recovering the spectral curve construction. 
% \begin{prop}\label{thm HKR}
%  There exists a section (called the Hitchin--Kostant--Rallis section or for short HKR section) to the 
%  Hitchin map $h$.
%  
%  This section takes values in the smooth locus $\moduli{\SU(p+1,p)}^{smooth}$ and has Toledo 
%  invariant $\tau=0$.
%  Moreover, it factors through 
%  $\modulir{\Spin(n+1,n)}$.
%  \end{prop} 
% \proof See \cite{Tesis}.
\section{Cameral data for $\SU(p+1,p)$-Higgs bundles}
In this section we compute tha cameral data for $\SU(p+1,p)$-Higgs bundles using the results in Section \ref{section cameral data} to $\SU(p+1,p)$. For simplicity, 
we  realize  $\SU(p+1,p)$ as the subgroup of fixed points whose Cartan involution is  $\theta$ (cf. Appendix \ref{section appendix}).

\subsection{Cameral covers}
Let  $\omega\in B_{\SU(p+1,p)}$. To define the associated cameral cover (\ref{eq cam cover}),  
let
$$
\ld^\C\otimes K\cong K^{\oplus 2p}\cong\{(l_1,\dots,l_{2p+1})\in K^{\oplus 2p+1}\ :\ \sum l_i=0\};
$$
then, we have the projection 
$$
\ld^\C\otimes K\to\ld^\C\otimes K/W,\qquad (l_1,\dots,l_{2p+1})\mapsto (\sigma_1(l_i),\dots,\sigma_{2p+1}(l_i))
$$
where $\sigma_i$ denotes the $i$-th symmetric polynomial in $2p+1$ variables.

In particular, the \textbf{cameral cover} associated to $\omega\in B_{\SU(p+1,p)}$ is
\begin{equation}\label{eq cam cover}
\specc_\omega=\left\{(\lambda_1,\dots, \lambda_{2p+1})\in K^{\oplus 2p+1}\ \left|\ \begin{array}{l}
                                                                     \sum_k\lambda_k=0,\\
                                                                    \sigma_{2i}(\lambda_j)=\omega_i,\\
                                                                    \sigma_{2i+1}(\lambda_j)=0.
                                                                    \end{array}\right.\right\}
\end{equation}
As for $\specc_{SU(p+1,p)}$ (\ref{eq real cam}), it corresponds to the subscheme
\begin{equation}\label{eq real cam su} 
\specc_{SU(p+1,p)}=\left\{(\lambda_1,\dots, \lambda_{2p+1})\in \specc_\omega
\ \left|\ \begin{array}{l}
          \lambda_i=-\lambda_{p+1+i},\\
          \lambda_{p+1}=0.\end{array}\right.\right\}.
\end{equation}
Note that the vanishing of $\sigma_{2p+1}(\lambda_1,\dots,\lambda_{2p+1})$ distinguishes $2p+1$
Galois subcovers determined
by $\specc_i=\{\lambda_i=0\}$ whose Galois group is 
$S_{2p}\subset W=S_{2p+1}$. In particular, $\specc_{SU(p+1,p)}\subset\specc_{p+1}$. 
\begin{rk}\label{rk XSU}
For any $(l_i)\in \specc_\omega$, $\lambda\in\{l_i\ :\ i\}$ if and only if $-\lambda\in\{l_i\ :\ i\}$,
so it follows that generically over $B_{\SU(p+1,p)}$, $\specc_{SU(p+1,p)}$ is 
irreducible and all other irreducible components in $\specc_{p+1}$ are obtained by translating $\specc_{SU(p+1,p)}$ by elements
of $W/W(\la^\C)=S_{2p+1}/S_p\ltimes\Z_2^p$.  
\end{rk}
\begin{ex} 
In the rank one case, $\SU(2,1)$, the projection $K\oplus K\to K^3\oplus K^2\cong\ld^\C\otimes K/W$ reads
$$
(l,l')\cong \left(\begin{array}{ccc}
                     l+l'&0&0\\
                     0&-2l&0\\
                     0&0&l-l'
                    \end{array}
                    \right)\mapsto (l^2-(l')^2-4l'l,l'((l')^2-l^2))
$$
Thus, any cameral cover corresponding to a real Higgs bundle (with corresponding point of the Hitchin base
$\omega\in K^2$) satisfies 
$$l'((l')^2-l^2)=0.$$
Namely, we have three
subcovers
$$
\specc_1=\{l^2=\omega,l'=0\},\quad \specc_2=\{l'=l, -4l^2=\omega\},
$$
$$
\specc_3=\{l'=-l, 4l^2=\omega\}.
$$
Note that $\specc_1=\specc_{\SU(2,1)}$.
All three are double covers, with involutions induced by elements of the Weyl group:
$(1,3)\in S_3$ restricts to the cover involution
on $\specc_1$, as so do $(1,2)$ on $\specc_2$ and $(2,3)$ on $\specc_3$. 
 \end{ex}
\subsection{Cameral data}\label{section cam data}
Applying Theorem \ref{thm cameral global}, we have a correspondence between everywhere regular $\SU(p+1,p)$-Higgs bundles $(E,\phi)$
whose image via $h_{\SU(p+1,p)}$ is $\omega\in B_{\SU(p+1,p)}$ (note that no polystability condition
is assumed) and 
$(\C^\times)^{2p}$-principal bundles  $P\to\specc_\omega$ satisfying the equivariance conditions: 
\begin{equation}\label{eq cocyclic rels}
(i,j)^*(P\otimes \mc{R}_{ij})^{(i,j)}\cong P,\qquad  P|_{\specc_{SU(p+1,p)}}\cong P|_{\specc_{SU(p+1,p)}}^{w_0}\otimes R_{w_0}. 
\end{equation}
In the case under consideration, regularity just means that eigenspaces for the standard representation have dimension one. In the semisimple case, this is just the fact that eigenvalues are different. As for non 
semisimple elements, it means that the nilpotent part is $\SL(2p+1,\C)$-conjugate to an element with $1$'s over the diagonal on a block of the matrix. As for the semisimple part, 
it should have different eigenvalues.

By definition, the involution $\theta$ acts as multiplication by  $-1$ on $\m^\C$ and $+1$ on $\h^\C$, so that the element $w_0\in W$ defined as in condition CD2 in Theorem \ref{thm cameral global}
has the form
\begin{equation}\label{eq w0}
 w_0=\prod_{i=1}^p(i,p+1+i).
\end{equation}
 \begin{thm}\label{prop gric cam data}
 Let $\omega\in B_{SU(p+1,p)}$ be such that $\specc_{SU(p+1,p)}$ is smooth. Then, the choice of a cameral datum
 $P_0$ establishes a correspondence between regular
 $\SU(p+1,p)$-Higgs bundles mapping to $\omega$ and the subvariety of $H^1(\specc/\theta,T^\C)$ consisting of elements 
 $Q$ such that $w^*q^*Q(D^\C)^w\cong Q(D^\C)$ for all $w\in W$. Here $\specc\stackrel{q}{\to}\specc/\theta$ is the 
 quotient curve, and $Q(D^\C)^w=Q(D^\C)\times w D^\C$.
 \end{thm}
\begin{proof}
 First note that the assignation
$$
P\mapsto PP_0^{-1}
$$
establishes a morphism $Cam\to H^1(\specc,D^\C)^W$, where the action of $W$ on $H^1(\specc, D^\C)$ 
is given by $w\cdot Q=w^*Q\times_wD^\C$.

Denote by $\specc_w=w\cdot \specc_{\SU(p+1,p)}$ (\ref{eq real cam su}) $Q_w:=Q|_{\specc_w}$. Then: 
$$
\theta^*Q_w=\theta^*w^*Q_e^w\stackrel{(\ref{eq cocyclic rels})}{=}\theta^*w^*Q_e^{w\theta}=Q_{w\theta}.
$$
Given that $\theta$ exchanges $\specc_w$ and $\specc_{w\theta}$, it follows that $\theta^*Q\cong Q$. 
Moreover, since $w_0^*Q=\theta^*Q\cong Q^{w_0}$, it follows that the structure group reduces to 
$T^\C$, as also $w_0^*Q^{w_0}\cong  Q$. This finishes the proof, as
the action of $\theta$ on $T^\C=(D^\C)^\theta$ is by definition the identity on the fixed locus of $w_0$, and so Kempf's descent Lemma (Theorem 2.3 \cite{DNPicard}) applies.
\end{proof}
\section{Spectral data}\label{section spectral}
In the case of matrix groups, Hitchin's spectral techniques are available for abelianization
of Higgs bundles. In this section, we compute spectral data for $\SU(p+1,p)$.

Let $(E,\phi)\in\moduli{\SU(p+1,p)}$ be such that $h(\phi)=(\omega_2,\dots,\omega_{2p})$. 
The characteristic
polynomial of $\phi$ over the total space of the canonical bundle $\pi:|K|\to X$ produces a section 
$$
s_\omega:=\lambda(\lambda^{2p}+\pi^*\omega_2\lambda^{2p-2}+\dots+\pi^*\omega^{2p})\in H^0(|K|,\pi^*K^{2p+1})
$$
vanishing over the \textbf{spectral curve} 
\begin{equation}\label{eq spectral}
\spec:=Spec\left(\mbox{Sym}^\bullet\left(K^*/\lambda(\lambda^{2p}+\sum_i\lambda^{2(p-i)}\pi^*\omega_{2i})\right)\right).
\end{equation}
The generic spectral curve $\spec$ is reduced and consists of two smooth irreducible components 
\begin{eqnarray}\label{eq irred comps}
\spec_0\cong Spec\left(\mbox{Sym}^\bullet(K^*/\lambda)\right)\cong X,\\\nonumber
\spec_1:=Spec(\mbox{Sym}^\bullet(K^*)/(\lambda^{2p}+\sum_i\lambda^{2(p-i)}\omega_{2i})).
 \end{eqnarray}
 \begin{rk}\label{rk genericity hypo}
Remark \ref{rk XSU} implies that the genericity hypothesis for $\spec$ to be of the above form is the same as the one for 
$\specc_{\SU(p+1,p)}$ being smooth,  and also for generic regularity of $(E,\phi)$.  
\end{rk}
By Remark \ref{rk genericity hypo}, the kernel of the Higgs field has rank one, so $V\subset E$ is an extension 
\begin{equation}\label{eq ext kernel}
0\to E_0\to V\to V_1\to 0 
\end{equation}
where $E_0=\Ker(\phi)\in \mathrm{Pic}(X)$; the Higgs field induces one on $E_1=V_1\oplus W$, so that we obtain $(E_1,\phi_1)$ an induced $\U(p,p)$-Higgs bundle,
which is regular by Remark \ref{rk genericity hypo} and a general result of Ng\^o's \cite{NgoLemme} and Arinkin 
(private communication). 
Moreover, regularity implies that
$\phi_1=(\beta_1,\gamma_1)$ induces generic isomorphisms
$$
\beta_1:W\to V_1\otimes K, \qquad \gamma_1:V_1\to W\otimes K.
$$
Taking determinants, we obtain subdivisors of the branching locus 
$B=\{\omega_{2p}=0\}$
\begin{equation}
 \label{eq subdivisor branching}
B_\beta=(s_\beta), \ B_\gamma=(s_\gamma)
\end{equation}
given by
vanishing of 
$$
s_\gamma:=\wedge^p\gamma_1\in H^0(X,\wedge^pV_1^{-1}\otimes \wedge^pW\otimes K^{p})
$$ 
and
$$
s_\beta:=\wedge^p\beta_1\in H^0(X,\wedge^pW^{-1}\otimes \wedge^pV_1\otimes K^{p}).
$$
\begin{rk}\label{rk divisors}
 The ramification divisor consists of points over which the Higgs field $\phi$ is not semisimple. Let $B_\beta^{reg}\subset B_\beta$ denote the
 subdivisor over which $\phi$ is regular. That is, $\beta|_{B_\beta^{reg}}:W\projects E_0K$, as otherwise the kernel 
 would increase its dimension by one. We let 
 $B_\beta^0:=B_\beta\setminus B_\beta^{reg}$.
 
 As for $B_\gamma$, the genericity hypothesis implies it is disjoint from $B_\beta$. Denote by
 $B_\gamma^{reg}\subset B_\gamma$ the subset over which $\phi$ remains regular. This implies that 
 $E_0\otimes \mc{O}_{B_\gamma^{reg}}\plonge V_1\otimes \mc{O}_{B_\gamma^{reg}}$, or, in other words, so that the kernel has still rank $1$. Let
 $B_\gamma^0:=B_\gamma\setminus B_\gamma^{reg}$.
\end{rk}
 
 From now on we will assume that
 \begin{equation}\label{eqn B=Breg}
B_{\beta/\gamma}=B_{\beta/\gamma}^{reg},  
 \end{equation}
so that the Higgs field is everywhere regular and $\beta|_{B_\beta}:W|_{B_\beta}\projects E_0K|_{B_\beta}$ over the whole
 $B_\beta$ and  $i:E_0|_{B_\gamma}\plonge V_1|_{B_\gamma}$ over ${B_\gamma}$. See Section \ref{section non reg} for remarks on 
 the general case.
By Remark \ref{rk divisors}, we have that the extension $[V]\in H^1(X,E_0^{-1}V_1)$ defined by (\ref{eq ext kernel}) 
 can be recovered as the image
via the   Bockstein map of the inclusion
$f\in H^0(X,E_0^{-1}V_1\otimes\mc{O}_{B_\gamma})$.
We recall that the Bockstein operator
\begin{equation}\label{eq bockstein}
 b:H^0(B_\gamma,E_0^{-1}V_1)\to H^1(X,E_0^{-1}V_1),\qquad f\mapsto [V],
\end{equation}
is obtained by considering the long exact sequence induced from the short exact sequence
$$
\xymatrix{
0\ar[r]&E_0^{-1}V_1\ar[r]& E_0^{-1}V_1(B_\gamma)\ar[r]&E_0^{-1}V_1\otimes O_{B_\gamma}\ar[r]&0.
}
$$
We next give an alternative approach to $\phi$. Consider the commutative diagram of extensions: 
$$
\xymatrix{
0\ar[r]&W^{-1}E_0K(-B_\beta)\ar[r]& W^{-1}E_0K\ar[r]\ar[d]&W^{-1}E_0K\otimes O_{B_\beta}\ar[d]\ar[r]&0\\
0\ar[r]&W^{-1}VK(-B_\beta)\ar[r]&W^{-1}VK\ar[r]\ar[d]&W^{-1}VK\otimes O_{B_\beta}\ar[d]\ar[r]&0\\
0\ar[r]&W^{-1}V_1K(-B_\beta)\ar[r]& W^{-1}V_1K\ar[r]&W^{-1}V_1K\otimes O_{B_\beta}\ar[r]&0.
}
$$
Studying the corresponding long exact sequences, we see that a necessary and sufficient condition 
for $\beta_1\in H^0(X,W^{-1}V_1K)$ and $p\in H^0(X, W^{-1}E_0K\mc{O}_{B_\beta})$
to determine $\beta\in H^0(X,W^{-1}VK)$ is
$$
p\in \Ker\left(H^0(X,W^{-1}E_0K\otimes O_{B_\beta})\to H^1(X,W^{-1}E_0K(-B_\beta))\right)
$$
and
$$
\beta_1\in \Ker\left(H^0(X,W^{-1}V_1K)\to H^1(X,W^{-1}E_0K)\right).
$$
This proves the following
\begin{prop}\label{prop char Higgs}
 An everywhere regular $\SU(p+1,p)$-Higgs bundle $(E,\phi)$ is equivalent to the following piece of data:
 
 1. An everywhere regular $\U(p,p)$-Higgs bundle $(E_1,\phi_1)$, where 
 $$E_1=V_1\oplus W$$
 and 
 $$
 \phi_1=(\beta_1,\gamma_1),\ \beta_1:W\to V_1K,\ \gamma_1:V_1\to WK.
 $$
 2. An embedding $i:H^0(B_\gamma,E_0^{-1}V_1)$, where $E_0=(\det W\det V_1)^{-1}$ and $B_\gamma$ is given by 
 (\ref{eq subdivisor branching}) determining an extension $V\in H^1(X,E_0^{-1}V_1)$.
 
 3. A map 
 \begin{equation}\label{eq conds p}
    p\in \Ker\left(H^0(X,W^{-1}E_0K\otimes O_{B_\gamma})\to H^1(X,W^{-1}E_0K(-B_\gamma))\right)
 \end{equation}
 
 4. The Higgs field $\phi_1$ should satisfy
\begin{equation}\label{eq conds on phi1}
\phi_1\in \Ker\left(H^0(X,\End(E_1)\otimes K)\to H^1(X,\Hom (E_1,E_0)K)\right).
\end{equation}
\end{prop}

Now, the involution $\theta'$ (see (\ref{eq theta})) induces an involution on $\spec$ sending $\lambda\mapsto-\lambda$.
 Let 
\begin{equation}\label{eq quotient spectral}
\ol{X}=\spec/\theta,
\end{equation}
 whose irreducible components read
\begin{equation}\label{eq irred comps quotient spectral}
 \ol{X}_0\cong X ,\qquad \ol{X}_1\cong \spec_1/\theta.
\end{equation}
We have a diagram
\begin{equation}\label{eq diagram curves}
\xymatrix{
\spec\ar[r]^{\tilde{\pi}}\ar[dr]_{\pi}&\ol{X}\ar[d]^{\ol{\pi}}\\
&X.
} 
\end{equation}
The ramification divisor of $\tilde{\pi}: \spec\to\ol{X}$
is given by $R=\{\pi^*\omega_{2p}=0\}$ and generically equals $\spec_0\cap\spec_1$. 
The divisor $R$ is also the ramification divisor of $\widetilde{\pi}|_{\spec_1}$.
Let $F\in\Pic(\spec)$ be the line bundle defined by:
\begin{equation}\label{eq spectral datum}
0\to F(-R)\to \widetilde{\pi}^*E\stackrel{\lambda(\lambda Id-\widetilde{\pi}^*\phi)}{\to}\widetilde{\pi}^*(E\otimes K)\to F\otimes\widetilde{\pi}^*K\to 0. 
\end{equation}
\begin{rk}\label{rk line bunds on irr curves}
 We need only remark that the torsor structure of the fibers as described in Theorem \ref{thm spectral data} comes from the decomposition of $\Pic(\spec)$ as a torsor over
$\Pic(\spec^n)$, the Picard variety of its normalization, which in this case is just $\spec^n=\spec_0\sqcup\spec_1\stackrel{p}{\to}\spec$. To see this, consider the short exact sequence
\begin{equation}\label{eq SES}
 0\to\mc{O}_{\spec}^\times\plonge p_*\mc{O}_{\spec^n}^\times\projects p_*\mc{O}_{\spec^n}^\times/\mc{O}_{\spec}^\times\to 0
\end{equation}
which induces a long exact sequence in cohomology
\begin{equation}\label{eq LES}
0\to H^0(\spec, \mc{O}_{\spec}^\times)\to H^0(\spec, p_*\mc{O}_{\spec^n}^\times)\to
H^0(\spec^{sing}, p_*\mc{O}_{\spec^n}^\times/ \mc{O}_{\spec}^\times)\to
\end{equation}
$$
\phantom{0}\to H^1(\spec, \mc{O}_{\spec}^\times)\to H^1(\spec, p_*\mc{O}_{\spec^n}^\times)\to 0
\phantom{H^0(\spec^{sing}, p_*\mc{O}_{\spec^n}^\times)}.
$$
See Section 9.2 in \cite{Neron} for details. 

Geometrically, line bundles on $\spec$ are given by line bundles on its normalization with suitable automorphisms
over the singular points. In our case, the $\mc{O}_{\spec}$-module structure of the line bundle $F\to\spec$ consists of an $\mc{O}_{\spec_i}$-module structure on $F_i$ on the respective
irreducible components together with an isomorphism $f:F_0|_{\spec_{01}}\cong F_1|_{\spec_{01}}$. Note that $R=\spec_{01}$ with our genericity hypothesis.
\end{rk}
Note that $F_0=F|_{\widetilde{X}_0}$ is the kernel of $\pi^*\phi$, and $F_1=F|_{\widetilde{X}_1}$ is the spectral datum for the induced $\U(p,p)$-Higgs
bundle 
$(E_1,\phi_1)$ (see discussion following Remark \ref{rk line bunds on irr curves}).  

It is easy to see that the spectral datum $F$ satisfies $\theta^*F\cong F$.

We next summarize some results of \cite{SchapUpp} for the group $\U(p,p)$, as the line bundle $F_1\to\spec_1$ (\ref{eq irred comps}) 
is the spectral datum corresponding
to the $\U(p,p)$-Higgs bundle
$(E_1,\phi_1)$. 
Define the subdivisors $R_+$ and $R_-$, where $R_+$ is the set of points 
over which
$\theta$ acts on the fibers of $F_1$ via the identity and $R_-$ the set of points where it lifts as multiplication by 
$-1$. 
 This determines
$\tilde{\pi}_*F_1\cong F_+\oplus F_-$, with $F_-^{-1}\otimes F_+\otimes \ol{\pi}^*K\cong\mc{O}(R_+)$.
As for $\phi_1$, the $\mc{O}_{\specc_1}=\tilde{\pi}^*\mc{O}_{\ol{X}_1}$-module structure of $F_1$ 
induces a $\tilde{\pi}_*\mc{O}_{\spec_1}=\mc{O}\oplus \ol{\pi}^*K^{-1}$-module structure on $\tilde{\pi}_*F_1$, 
totally determined by the action of 
$\ol{\pi}^*K^{-1}$. 
Since the involution $\theta$ acts by $-1$ on $\pi^*K$ (as by definition $\lambda\mapsto-\lambda$)
we get $s_{\pm}:\ol{\pi}^*K^{-1}\otimes F_{\pm}\to F_{\mp}$.
We have that $W=\ol{\pi}_*F_-$, $V_+=\ol{\pi}_*F_+$, $\phi_1=(\ol{\pi}_*s_+,\ol{\pi}_*s_-)$. 

We will use this to recover the $\SU(p+1,p)$-Higgs bundle. By Remark \ref{rk line bunds on irr curves},
the $\mc{O}_{\spec}$ structure of the line bundle $F\to\spec$ pushes forward to a
$\mc{O}_X$-module structure on $F_0$, and a $\mc{O}_{\ol{X}_1}\oplus\ol{\pi}^*K^{-1}=\widetilde{\pi}_*\mc{O}_{\spec_1}$-module
structure on $\widetilde{\pi}_*F_1=F_+\oplus F_-$. Away from ramification, the latter is generated by the action of
$\ol{\pi}^*K^{-1}$, yielding $s_{\pm}$. Denote $\ol{R}:=\widetilde{\pi}_*R$; since the involution is trivial on the irreducible
component $\spec_0$ over which the kernel lives and it becomes trivial on the quotient, 
we get 
the following $\widetilde{\pi}_*\mc{O}_{R}=\mc{O}_{\ol{R}}\oplus\ol{\pi}^*K^{-1}$-module structure 
on $\widetilde{\pi}_*F|_{\widetilde{R}}$: on the one hand, the usual action of
$\mc{O}_{\ol{R}}$, and the restriction of $s_\pm$ to ramification
$$
s_{\pm}:\ol{\pi}^*K^{-1}\otimes F_{\pm}\to F_{\mp}.
$$
On the other hand, isomorphisms induced from $f$ as in Remark \ref{rk line bunds on irr curves}:
$$
f_-:\ol{\pi}^*K^{-1}\otimes F_-\cong F_0
$$
on $\ol{R}_-$, and
$$
f_+: F_+\cong F_0.
$$
on $\ol{R}_+$.
\begin{rk}
By Remark \ref{rk divisors}, we could also consider just morphisms, corresponding to generically (but not everywhere)
regular Higgs bundles.
\end{rk}
\begin{rk}\label{rk Rpm vs Rbetagamma}
Clearly $\ol{\pi}_*R_+=\ol{\pi}^*B_\gamma$, $\ol{\pi}_*R_-=\ol{\pi}^*B_\beta$.% and likewise for the regular loci.
\end{rk}
The above discussion suggests the following.
\begin{thm}\label{thm spectral data}
Let $\omega\in B_{SU(p+1,p)}$ be generic. Let $\pi:\spec_\omega\to X$ be the corresponding spectral curve, and
$\ol{\pi}:\ol{X}_\omega=\widetilde{X}_\omega/\theta\to X$. Denote by $\widetilde{X}_{01}$/$\ol{X}_{01}$ the intersection of both 
irreducible components (\ref{eq irred comps quotient spectral}). 

Let
$$\Pic(\spec_1)^\theta=\{F_1\in\Pic(\spec_1)\ :\ \theta^*F_1\cong F_1\},
$$
and let $\mc{P}$ be the $(\C^\times)^{4p(g-1)}$-torsor over 
with fiber over $F_1$ equal to $\Isom(\spec_{01}, L_0^{-1}F_1)$, with
\begin{equation}\label{eq L0 norm}
L_0=Nm(F_1)^{-1}\otimes K^{-p(2p+1)}\in\Pic(X). 
\end{equation}
With the notation of the preceeding paragraph:

1. There exists a correspondence between isomorphism classes of everywhere regular $\SU(p+1,p)$-Higgs bundles mapping via 
$h_{\SU(p+1,p)}$ to $\omega$ (note that we assume no stability
condition) and the elements of the quotient $\mc{P}/\C^\times$ where the action is given by
$$
\mu(L,\ol{z})\mapsto(L,\mu\ol{z}),\ (L,\ol{z})\in \mc{P}.
$$
2. Let $\widetilde{\pi}_*F_1=F_+\oplus F_-$. Then, the topological invariant (\ref{eq MW}) is 
$$
\deg W=\deg F_--2p(p-1)(g-1).
$$

3. The corresponding Higgs bundle is stable, thus  
$$
(2p^2-3p)(g-1)< |\deg F_-|<(2p^2-p)(g-1),
$$
and extremal values of the Milnor-Wood inequality (\ref{eq MW}) $\deg W=\pm p(g-1)$ are not met by regular points. 
\end{thm}
\begin{proof}
   To prove \textit{1.}, given $(E,\phi)$, we assign to it $F$ as in (\ref{eq spectral datum}). Generically, the spectral curve has smooth irreducible components,
  and so restriction induces on its Picard variety a structure of a $(\C^\times)^{4p(g-1)}$-torsor over 
  $\Pic(\spec_0)\times\Pic(\spec_1)$. See Remark \ref{rk line bunds on irr curves}.
  The topological restriction that $\det\pi_*F=0$
  implies (\ref{eq L0 norm}), by Corollary 3.12 and Lemma 3.5 in \cite{HP}.
  
%   
%   By Remarks \ref{rk divisors}, \ref{rk Rpm vs Rbetagamma} and \ref{rk line bunds on irr curves}, regular points correspond to bundles for which $f$ is an isomorphisn, 
%   and non-everywhere regular 
%   bundles to  $f$ vanishing at a subdivisor of ramification.
%   
  As for the inverse map, pushforward induces the remaining structure as explained in the discussion preceeding this theorem.
  There remains to check (a) that the construction descends to the quotient by the action of $\C^\times$, and (b) that conditions (\ref{eq conds on phi1}) and (\ref{eq conds p}) 
  in Proposition \ref{prop char Higgs} are satisfied for 
  $$
  \beta_1:=\ol{\pi}_*s_-\in H^0(X,\Hom(\ol{\pi}_*F_-\otimes \ol{\pi}_*F_+\otimes K)
  $$ and 
  $$
  p:=\ol{\pi}_*f_-\in H^0(X,\Hom(\ol{\pi}_*F_-, E_0\otimes K)
  $$
  Statement (b) follows by the fact that $\Hom(F_0, F_1)$ is a sheaf supported on $\spec_{01}$, thus 
  $H^1(\spec_{01},\Hom(F_0, F_1))=0$. As for (a),  an $\SU(p+1,p)$-Higgs bundle is determined by point  $L:=(L_0, L_1, f)\in \Jac(\spec)$, with the restriction that the degrees
  of $L_0$ and $L_1$ are such that the pushforward has degree zero. Now, for $\lambda\in \C^\times$ note that the morphism given by multiplication of $\lambda^{-1/2}$ on
  $L_0$ and $\lambda^{1/2}$ on $L_1$ send $(L_0,L_1,f)\mapsto (L_0,L_1,\lambda f)$; this yields isomorphic Higgs bundles.
  
Statement  \textit{2.} follows from $\det\ol{\pi}_* F_-=Nm (F_-)\otimes Nm(\ol{X}/X)$ together with $Nm(\ol{X}/X)=\bigotimes_{i=0}^{p-1}K^{-2i}=K^{-p(p-1)}$.

Finally, \textit{3.} follows because by irreducibility of $\spec_1$ and regularity, the only non-trivial 
$\phi$-stable subbundle is the kernel; by definition, for $B_\gamma$ as in (\ref{eq subdivisor branching}),  
$\mc{O}(B_\gamma)\cong E_0^{-1}\det W^{-2}K^p$, so it follows that 
$\deg E_0=2\deg W-2p(g-1)-\deg B_\gamma$, which is strictly smaller than $0$ by point \textit{2.} above and because  $B_\gamma\neq\emptyset$
by regularity. As for the strictness of the inequality, it follows from Theorem 6.7 in \cite{BGGUpq}.
\end{proof}
\begin{rk}\label{rk stability bundles}
We note that by reducibility of the spectral curve, there are always unstable Higgs bundles mapping to a point 
of the 
Hitchin base. Indeed, take any direct sum $(E_0, 0)\oplus (E_1,\phi_1)$, where $(E_1,\phi_1)$ is a 
$\U(p,p)$-Higgs bundle of 
degree  $\deg E_1=-\deg E_0$ mapping to $\omega\in \hbase_{\U(p,p)}$. Making $\deg E_0>0$ 
suffices to get
an unstable point. 

An interesting consequence of this is the lack of intrinsicity of the Milnor-Wood inequality with respect
to the spectral data, unlike what happens for 
$\U(p,p)$-Higgs bundles \cite{SchapUpp}, which is why point 2. in Theorem \ref{thm spectral data} is necessary.

The same phenomenon will show for all  forms of Hermitian non-tube type, as they all contain a 
maximal tube-type subgroup
of the same rank.
\end{rk}
\begin{rk}
The set we recover has the expected dimension. By the preceeding dicussion, using the sequence (\ref{eq LES}) we have
that the generic fiber $F_\omega$ has dimension $\dim F_\omega=\dim \Pic(\spec_1)^\theta+4p(g-1)-1$. Given that 
$\Pic(\spec_1)^\theta$ is a smooth fiber
of the Hitchin map for $\U(p,p)$,  and  by Theorem \ref{thm spectral data} 
$$
\dim F_\omega=4p(g-1)+(4p^2(g-1)+1-\dim B_{\SU(p,p)})-1,
$$
so that the set of all generic fibers has dimension
$$
\dim B_{\SU(p+1,p)}+\dim F_\omega=4p(p+1)(g-1)=\dim \SU(p+1,p)(g-1).
$$
\end{rk}

We can give another characterization in terms of data over $\ol{X}$:
\begin{cor}\label{prop another spec data}
There is a a one to one correspondence between  
isomorphism classes of regular Higgs bundles $(E,\phi)$ mapping to $\omega$ via $h_{\SU(p+1,p))}$ 
and the quotient by the action of $\C^\times$ on tuples 
$$
(F_-,s_+,s_-,f_+,f_-)
$$
where
1.  A line bundle $F_+\in \mathrm{Pic}(\ol{X}_1)$, subject to condition in 4. below.

2. Morphisms 
\begin{eqnarray}\label{eq conds spm}
s_+\in \Ker\left(H^0(\ol{X},F_+^{-1}F_-\ol{\pi}^*K)\to H^1(X,F_+^{-1}\ol{\pi}^*(E_0K))\right),\\\nonumber
s_-\in \Ker\left(H^0(\ol{X},F_-^{-1}F_+\ol{\pi}^*K)\to H^1(X,F_-^{-1}\ol{\pi}^*(E_0K))\right). 
\end{eqnarray}
where $R_\pm=(s_\pm)$, 
 $F_-:=\mc{O}(-R_-)F_+\ol{\pi}^*K$ and $E_0=\mathrm{Nm} F_+^{-2}\otimes O(-\pi_*R_+)K^{-p(2p+1)} \in \mathrm{Pic}(X)$.

3. Isomorphisms 
\begin{eqnarray}\label{eq fpm}
f_-\in\mathrm{Isom}(E_0|_{R_-},F_+|_{R_-}),\\\nonumber
f_+\in \mathrm{Isom}(F_+|_{R_+},E_0|_{R_+}\otimes \ol{\pi}^*K). 
\end{eqnarray}

4. The bundle $F_+$ should satisfy that
$F_+\oplus F_+\otimes\mc{O}(-R_-)$ be a $\widetilde{\pi}_*\mc{O}_{\spec}$ module.

5. Moreover 
$$
\deg W=\deg F_+-(2p^2-2p+1)(g-1).
$$
The corresponding Higgs bundle is stable and
$$
(2p^2-3p)(g-1)< |\deg F_- |<(2p^2-p)(g-1).
$$
\end{cor}
\begin{proof}
 It follows from Theorem \ref{thm spectral data} observing that pushforward to $\ol{X}$ produces the objects 
in this corollary. 
 \end{proof}
Regarding connected components, we have
\begin{cor}\label{cor con comps}
 The generic regular fiber has connected components classified by values of $\deg W$ other than $\pm p(g-1)$. 
\end{cor}
 
% \begin{rk}
% In other words, we have a fibration
% $$
% h_{\SU(p+1,p)}^{-1}(\omega)\to h_{\U(p,p)}^{-1}(\omega)
% $$ 
% whose fiber over $(L_+,s_+,s_-, R_+,R_-)$ consists of the objects $(f_+,f_-)$ specified in point 3. of 
% Theorem \ref{thm spectral data}.
% \end{rk}
% 
%  $$
% 0\to F\to E_0\oplus L\to L|_R\to 0,\qquad (e,l)\mapsto f(e|_{R})-l|_{R}.
% $$
% Pushing forward to $\ol{X}$ we obtain isomorphism $f_+: E_0|_{R_+}\cong L_+|_{R_+}$ and 
% $f_-: E_0|_{R_-}\cong L_-|_{R_-}$. Now, pushforward being left exact, $f_+$ induces an embedding
% $$
% i:E_0|_{R_+}\plonge \ol{\pi}_*L_+,
% $$
% which allows to recover $[V]\in H^0(X,E_0^{-1}\ol{\pi}_*L_+)$.
% 
% Likewise, we recover $(E,\phi)$ from the above, we note that we have a short exact sequence
% $$
% 0\to E_0\to V\oplus W\to V_+\oplus W\to 0
% $$
% induced from $\phi$. Clearly, $\phi|_{E_0}\equiv 0$. On the other hand $\phi: W\to V$Or equivalently, 
% $$
% 0\to \mc{O}\to E_0^{-1}(V\oplus W)\to E_0^{-1}(V_+\oplus W)\to 0,
% $$
% whose class in $H^1(X,E_0^{-1}(V\oplus W))$ is determined as follows: over $B_+$ we have
% $E_0\cong \mc{O}(B_+)(\det W)^{-2}K^{-p}\subset V$. Consider the short exact sequence
% $$
% 0\to H^0(X,E_0^{-1}V_1)\stackrel{s_+}{\to} H^0(X,\mc{O}(-B_+)(\det W)^{2}K^{p}V_1)\to 
% H^0(D_+,X,\mc{O}(-B_+)(\det W)^{2}K^{p}V_1)
% $$
\section{From cameral to spectral data}\label{section from cam to spec}
As the name of the section suggest, we next explain how the spectral data can be obtained from the cameral data.

Following \cite{D93}, we associate to the cameral cover $\specc$ an intermediate cover (isomorphic to the spectral cover) associated to the standard representation.

To do this, we consider the sistem of simple roots 
$$
S=\{\alpha_i=L_i-L_{i+1},\  i=1,\dots, 2p-1\},
$$
where where $L_i$ applied to a diagonal matrix returns the $i$-th entry and $\alpha_i>\alpha_{i+1}$. This ordering 
satisfies condition (\ref{condition order}) and so we may choose a compatible Borel subgroup $B<\SU(p+1,p)$.

Let $\delta_1$ be the first fundamental weight, which is the highest weight of the standard representation, 
and its associated maximal parabolic $P_{\delta_1}$. 
Let $W_{\delta_1}<W$ be the corresponding Weyl group. Then $W_{\delta_1}\cong S_{2p}$. We have maps
$$
\xymatrix{
\specc\ar[r]\ar[dr]&\specc_{\delta_1}:=\specc/W_{\delta_1}\ar[d]\\
&X}
$$
and moreover, by the discussion in \cite{D93}  \S 4,  $\specc_{\delta_1}$ is isomorphic to the spectral cover 
(\ref{eq spectral})
via the morphism:
\begin{eqnarray}\label{eq map cam to spec}
 X\times_{\ld^\C\otimes K/W}\ld^\C\otimes K&\to& K\\\nonumber
 (x,q)&\mapsto&(x,\delta_1(q))
\end{eqnarray}
Identifying $\specc_{\delta_1}$ and $\spec$, we have a diagram
\begin{equation}\label{eq spec as quotient}
 \xymatrix{
\specc\ar[r]^{\widehat{\pi}}\ar[dr]_{p}&\spec\ar[d]^{\pi}\\
&X.
}
\end{equation}
\begin{lm}\label{lm descent theta}
The action of $\theta$ on $\specc$ descends to the involution on $\spec_{\delta_1}$ sending 
$(x,\lambda)\mapsto(x,-\lambda)$.
\end{lm}
\begin{proof}
Let $s_i$ denote the reflection associated to the simple root $\alpha_i$. From (\ref{eq w0}) we check that
$(1,2)\circ\theta\in W_{\delta_1}$, so that both maps induce the same one on $\spec_{\delta_1}$. Since
  $\delta_1=\check{\alpha_1}^*\in(\ld^\C)^*$, it follows that $-\delta_1=\delta_1\circ s_1$, whence the result.
\end{proof}
\begin{prop}\label{prop descent}
Let $\omega\in B_{\SU(p+1,p)}$ be generic, and let $P_0\to\specc_\omega$ be a cameral datum. Let $Cam(\specc)$ denote the variety of isomorphism classes of cameral
data. 
Then $Cam(\specc)\cong \{L\in \Pic(\spec)\ : \ \theta^*L\cong L, \deg L_0=-\deg L_1-2p(2p+1)(g-1)\}$. 
\end{prop}
\begin{proof}
Proceeding as in Theorem \ref{prop gric cam data}, we assign
$$
P\mapsto PP_0^{-1}
$$
thus establishing an isomorphism $Cam(\specc)\to H^1(\specc,D^\C)^{W,\theta}$, with the induced action (and conditions on it) of $W$ and
$\theta$, that is  $w\cdot Q=w^*Q\times_wD^\C$, and $\theta^* Q|_{\specc_{\SU}}\cong Q|_{\specc_{\SU}}$.

Now, given $Q\in H^1(\specc,D^\C)$, then $L_Q:=Q\times_{\delta_1}\C^\times$ is a line bundle satisfying $w^*L_Q\cong L_Q$ for all 
$w\in W_{\delta_1}$ (by Proposition 5.5 in \cite{DG} and the fact that the action of $w\in W_{\delta_1}$ on the fibers 
is trivilised by composing with $\delta_1$). The same reasons imply that  the action of $w$ over its associated ramification 
divisor is trivial. 
So by Kempf's Descent Lemma, $L_Q$ descends to $\tilde{L}_Q\to\spec$. As for equivariance $\theta^*\tilde{L}_Q\cong \tilde{L}_Q$, 
it follows from Lemma \ref{lm descent theta}.

The map is an isomorphism by Proposition 9.5 in \cite{DG} and Theorem \ref{prop gric cam data}.
\end{proof}

\section{Non-regular Higgs bundles}\label{section non reg}
Reducibility of the spectral curve causes for the existence of non-regular Higgs bundles over any given point of the base, which are not captured
by the cameral construction, as in fact they aren't intrinsic to the point of the base.

The explanation lies in Theorem 17.5 in \cite{DG}, according to which a Higgs bundle within a Hitchin fiber
is equivalent to a cameral datum together with the 
extra data of a $W$-equivariant morphism $F:\specc\to\ld^\C\otimes K$, regular fields are given by
embeddings, or projections onto the second factor $\specc:=X\times_{\ld^\C\otimes K/W}\to(\ld^\C\otimes K)$,
and hence are intrisic to the point of the base.

Now, such a morphism $F$ can be proved to be 
equivalent to fixing a Higgs field with values in a subsheaf of regular centralisers totally determined by the point
of the Hitchin base \cite{DG}. 
In Theorem \ref{thm spectral data}, the hypotheses on the point of the Hitchin base ensure that 
the Higgs field is completely determined away from ramification. 
Hence, it suffices to determine it over ramification. This is precisely the information encoded in $f_{\pm}$ in 
(\ref{eq fpm}) and requirement (\ref{eq conds spm}) in Corollay \ref{prop another spec data}.

When the Higgs field is not regular, $f_{\pm}$ are not isomorphisms anymore, but morphisms. 
The locus over which they vanish is 
 $B_{\beta/\gamma}^0$ defined in Remark \ref{rk divisors}. A part of these non-regular bundles can be produced by considering 
the a bundle of toric varieties over $Pic(\spec_1)^\theta$ with fiber $\Hom(\spec_{01}, L_0^{-1}F_1))$ over $F_1$ 
(and dense torus $\Isom(\spec_{01}, L_0^{-1}F_1))\subset \Hom(\spec_{01}, L_0^{-1}F_1))$, cf. 
Theorem \ref{thm spectral data}). These are coherent sheaves over $\spec$ whose restrictions to the irreducible 
components are locally free away from the singular locus. The pushforward of these is 
again a well defined Higgs bundle; nonetheless, the 
correspondence fails to be unique at this level, as the Higgs bundles thus produced will be those having
zero nilpotent
part over the ramification locus of the curve. In order to describe all of them it is necessary to introduce a stratification by ``degree of
regularity'' (measured by the dimension of the centralisers of the Higgs field over $B_{\beta/\gamma}^0$). We hope to address this 
questions in the near future.
\section{On regularity}\label{section regularity}
The relation between regularity (cf. Definition \ref{defi regular}) and smoothness of points of the complex Hitchin fiber 
essentially goes back to Kostant's \cite{Kos}, as it is proved by Biswas and Ramanan 
(\cite{BisRam}, Theorem 5.9). Their proof applies to the real group case, so we have:
\begin{prop}\label{Biswas Ramanan adapts}
If a $G$-Higgs bundle $(E,\phi)$ is a smooth point of $h_G^{-1}(\omega)$, then $\phi(x)\in\mr$ 
for all $x\in X$.
\end{prop}
\begin{proof}
 Let $x\in X$. We have
 that $\textrm{ev}_x\circ h_G(E,\phi)=\chi\phi_x $, where $\chi:\m^\C\to\lie{a}^\C//W(\la)$ is the Chevalley
 map. At a smooth point of the fiber, $dh_G$ is surjective, and since $\textrm{ev}_x$ is surjective
 too, it follows that $d(\chi\circ \textrm{ev}_x)$ is itself surjective. Since 
 $d\textrm{ev}_x:H^0(X,E(\m^\C\otimes K))\to \m^\C\otimes K_x$ is surjective, and is itself 
 evaluation
 at $x$, this implies that $d_{\phi_x}\chi$ is surjective.
 But Kostant--Rallis' work \cite{KR71} implies this happens if and only if $\phi_x$ is regular.
\end{proof}
\appendix
\section{Lie theory}\label{section appendix}
The subgroup $\SU(p+1,p)\leq \SL(2p+1,\C)$ is defined as the locus of fixed points of the involution
$$
 \sigma(X)=\Ad(I_{{p+1},{p}})\phantom{.}^t\ol{X}^{-1}
$$
where
$$
I_{{p+1},p}=\left(
\begin{array}{cc}
I_{p+1}&0\\
0&-I_{p}
\end{array}
\right).
$$
By composing with the compatible compact involution  $\tau(X)=^t\ol{X}^{-1}$, we obtain a linear involution 
\begin{equation}\label{eq theta}
 \theta'=\Ad(I_{p+1,p}).
\end{equation}
whose restriction to $\SU(p+1,p)$ is $\tau$. The differential of this involution (denoted also by $\theta'$) induces the Cartan decomposition
$$
\lie{su}(p+1,p)=\lie{s}(\lie{u}(p+1)\oplus \lie{u}(p))\oplus\m.
$$
Here, $\m$ is defined by
$$
\m=\left\{\left(
\begin{array}{cc}
0&B\\
C&0
\end{array}
\right)\in\lie{sl}(2p+1,\C)\ \left| B\in\ \mbox{Mat}_{p+1\times p}(\C),\ C=^t\ol{B}\right.
\right\}.
$$
To this decomposition it corresponds a polar decomposition on the level of the group. Indeed, $\SU(p+1,p)=He^\m$, where
$H=S(\U(p+1)\times \U(p))$ is realised as the subgroup of matrices of $\SL(2p+1,\C)$ of the form
$$
\left(\begin{array}{cc}
       A&0\\
       0&B^{-1}
      \end{array}
\right)
$$
where $A\in \U(p)$, $B\in \U(p+1)$, $\det A\det B=1$.

We next revise the theory of the isotropy representation necessary for this article. Recall that this representation
\begin{equation}\label{eq isotropy}
\iota:H^\C\to GL(\m^\C),
\end{equation}
is obtained by restriction of the adjoint representation.
\begin{defi}\label{defi max aniso}
 A maximal anisotropic Cartan subalgebra of $\sl(2p+1,\C)$ (associated with $\lie{su}(p+1,p)$) is the complexification of a maximal
 abelian subspace $\la'\subset\m$.
\end{defi}
One calculates easily (cf. \cite{K}, Chapter VI) that  the maximal anisotropic Cartan subalgebra ${(\la')^\C}$ 
consists of the matrices of the form
\begin{equation}\label{eq la}
M=\left(
\begin{array}{ccc}
0&0&0\\
0&0&X\\
0&\phantom{.}^tX&0
\end{array}
\right)
 \end{equation}
with $X\in\mbox{Mat}_{p\times p}(\C)$ antidiagonal.

The centraliser of ${(\la')^\C}$ in $H^\C$ is defined as the subgroup 
 $C_H({(\la')^\C})\subset H^\C$ defined as
 $$
 C_H({(\la')^\C})=\left\{h\in\ H^\C\ \left|\ Ad(h)A=A\qquad\textrm{ for any }A\in{(\la')^\C}\right.\right\}
 $$
 One readily checks that
\begin{equation}\label{eq CH}
 C_H({(\la')^\C})=\left\{\left(
 \begin{array}{ccc}
\det A^{-2}&0&0\\
 0&A&0\\
 0&0&A
 \end{array}
 \right)\ \left|\ A\in\mbox{GL}_{p\times p}(\C)\mbox{ diagonal }\right.\right\}.
\end{equation}
Note that the direct sum ${\ld'^\C}={(\la')^\C}\oplus\lc_\h({(\la')^\C})$, where $\lc_\h({(\la')^\C}):=\mbox{Lie}(\CH)$, is a $\theta'$-stable
Cartan subalgebra. Let $W$ be the corresponding Weyl group. 
Recall that we define the restricted Weyl group
\begin{equation}\label{eq res weyl}
W({(\la')^\C}):=\mbox{N}_W({(\la')^\C})/C_W({(\la')^\C}), 
\end{equation}
 where $N_W({(\la')^\C})$ is the
normaliser of ${(\la')^\C}$ in $W$ and $C_W({(\la')^\C})$ its stabiliser. This is isomorphic to the quotient
$\NH/\CH$. Chevalley' restriction theorem reduces the problem of studyin the quotien $\m^\C//H^\C$ to the much simpler 
$(\la')^\C/W((\la')^\C)$:
\begin{thm}[Chevalley's Restriction Theorem]\label{thm Chevalley}
Restriction $\C[\m^\C]\to\C[{(\la')^\C}]$ induces an isomorphism
$\C[\m^\C]\cong\C[{(\la')^\C}]^{W({(\la')^\C})}$.
\end{thm}

\begin{lm}\label{lm invars}
The algebra of polynomial invariants  $\C[{(\la')^\C}]^{W({(\la')^\C})}$ is isomorphic to $\oplus_i \mbox{Sym}^{2i}(\C)$. 
\end{lm}

% As for the normaliser
% $$
% N_H(\la)=\left\{\left(
% \begin{array}{ccc}
% x^{-2}&0&0\\
% 0&\epsilon x&0\\
% 0&0&\epsilon x
% \end{array}
% \right)\ \left|\ x\in \C^\times,\ \epsilon^2=1\right.\right\}
% $$
% So that we get the exact short sequence
% $$
% 0\to C_H(\la)\to N_H(\la)\to \Z/2\Z\to 0 
% $$
% where the second map is given by
% $$
% \left(
% \begin{array}{ccc}
% x^{-2}&0&0\\
% 0&\epsilon x&0\\
% 0&0&\epsilon x
% \end{array}
% \right)\mapsto  \frac{\epsilon x}{x}=\epsilon
% $$
% \begin{rk}\label{splits}
% The sequence splits in this case, and the splitting is given by
% $$
% -1\mapsto \left(
% \begin{array}{cccc}
% 1&0&0\\
% 0&-1&0\\
% 0&0&-1
% \end{array}
% \right)
% $$
% \end{rk}
% Note that the natural action of $\NH/\CH$ on $\CH$ is identically trivial.
% \begin{fact}
%  Nilpotent elements in $\m$ are elements of the form
%  $$
%  \left(
%  \begin{array}{ccc}
%   0&0&u\\
%   0&0&v\\
%   w&z&0
%  \end{array}
%  \right)
%  $$
%  satisfying $uw+vz=0$
% \end{fact}
\begin{defi}\label{defi regular}
 We define the subset $\mr\subset\m^\C$ of regular elements by
 \begin{equation}\label{eq mr}
 \mr=\{y\in\m^\C\ :\ \dim H^\C\cdot y\geq \dim  H^\C\cdot z\ \pt z\in \m^\C\}.
\end{equation}
Equivalently, $\dim\lc_{\m}(x)=\dim(\la')^\C$, where $\lc_{\m}(x)=\{y\in\m^\C\ :\ [y,x]=0\}$.
\end{defi}
\begin{lm}
 $\SU(2p+1,\C)<\SL(2p+1,\C)$ is a quasi-split form. In particular,
 $$
 \mr=\m^\C\cap\gr.
 $$
\end{lm}
\begin{proof}
 A form is quasi-split if and only if $\lc_\h({(\la')^\C})$ is abelian (cf. \cite{K}), which follows from \ref{eq CH}. In particular
 an element is regular if and only if its centraliser in $\g$ has the dimension of a Cartan subalgebra, as this is
 the case for semisimple elements. Thus it is regular
 in $\g$ if and only if it is regular in $\m$.
\end{proof}
 It is  useful to consider the following realization of $\su(p+1,p)$. Define the involution
 \begin{equation}\label{eq theta 2}
  \theta=\Ad(J_{p+1,p})
 \end{equation}
where $J_{p+1,p}$ is the matrix with $0$ entries everywehre except for the antidiagonal entries, which are $j_{i,2p-i+1}=1$ if 
$i\neq p+1$ and $j_{p+1,p+1}=-1$. Then the subalgebra $\ld^\C=\la^\C\oplus\lt^\C$,  with
\begin{equation}\label{eq la'}
 \la=\left\{ \left(  \begin{array}{ccc}
                      A&0&0\\
                      0&0&0\\
                      0&0&-A
                     \end{array}
\right)\ :\ A\in\mbox{Mat}_{p\times p}(\C)\mbox{ diagonal}\right\},
\end{equation}
\begin{equation}\label{eq lt'}
 \lt=\left\{ \left(  \begin{array}{ccc}
                      B&0&0\\
                      0&-2\mbox{tr}(B)&0\\
                      0&0&B
                     \end{array}
\right)\ :\ B\in\mbox{Mat}_{p\times p}(\C)\mbox{ diagonal}\right\}.
\end{equation}
is a $\theta$-invariant Cartan subalgebra which is also maximally 
anisotropic, as 
explained in \cite{GW}, Section 12.3.2.
\bibliographystyle{alpha}
\bibliography{SU.bib}

\begin{thebibliography}{BGG03}

\bibitem[BGG03]{BGGUpq}
S.B. Bradlow, O.~{Garc\'ia-Prada}, and P.B. Gothen.
\newblock {Surface group representations and {U}(p,q)-{Higgs} bundles}.
\newblock {\em J. Differential Geom.}, 64(1):117--170, 2003.

\bibitem[BLR90]{Neron}
S.~Bosch, W.~L{\"u}tkebohmert, and M.~Raynaud.
\newblock {\em {N\'eron Models}}, volume~21 of {\em {Ergebnisse der Mathematik
  und ihrer Grenzgebiete}}.
\newblock Springer-Verlag, 1990.

\bibitem[BNR89]{BNR}
A.~Beauville, M.S. Narasimhan, and S.~Ramanan.
\newblock {Spectral curves and the generalised theta divisor}.
\newblock {\em J. Reine Angew. Math.}, 398:169--179, 1989.

\bibitem[BR94]{BisRam}
I.~Biswas and S.~Ramanan.
\newblock {An infinitesimal study of the moduli of {Hitchin} pairs}.
\newblock {\em J. London. Math. Soc.}, 49(2):219--231, 1994.

\bibitem[Cor88]{Corlette}
K.~Corlette.
\newblock {Flat G-bundles with canonical metrics}.
\newblock {\em J. Diff. Geom.}, 28:361--382, 1988.

\bibitem[DG01]{DG}
R.Y. Donagi and D.~Gaitsgory.
\newblock {The gerbe of Higgs bundles}.
\newblock {\em Transform. Groups}, 7(2):109--153, 2001.

\bibitem[DN89]{DNPicard}
J.-M. Dr\'ezet and N.S. Narsimhan.
\newblock {Groupe de Picard de vari\'et\'es de modules de fibr\'es semistables
  sur les courbes alg\'ebriques}.
\newblock {\em Invent. Math.}, (97):53--94, 1989.

\bibitem[Don]{D95}
R.Y. Donagi.
\newblock {\em {Spectral covers}}, volume~28 of {\em {Math. Sci. Res. Inst.
  Publ.}}, pages 65--86.
\newblock Cambridge Univ. Press.

\bibitem[Don83]{Donaldson}
S.K. Donaldson.
\newblock {A new proof of a theorem of Narasimhan and Seshadri.}
\newblock {\em J. Differential Geom.}, 18(2):269--277, 1983.

\bibitem[Don93]{D93}
R.Y. Donagi.
\newblock {\em {Decomposition of spectral covers}}, volume 218, pages 145--175.
\newblock Ast{\'e}risque, 1993.

\bibitem[DP12]{DPLanglands}
R.Y. Donagi and T.~Pantev.
\newblock {Langlands duality for Hitchin systems}.
\newblock {\em Invent.Math}, (189):653--735, 2012.

\bibitem[FG12]{Emily}
E.~Franco-G{\'o}mez.
\newblock {\em {Higgs bundles over elliptic curves}}.
\newblock PhD thesis, {Universidad Aut{\'o}noma de Madrid}, 2012.

\bibitem[FGN]{FGN}
Higgs bundles over elliptic curves.
\newblock arXiv:1302.2881.

\bibitem[GGM]{GGMHitchinKobayashi}
O.~{Garc{\'i}a-Prada}, P.B. Gothen, and {Mundet-i-Riera, I.}
\newblock {The {Hitchin-Kobayashi} correspondence, {Higgs} pairs and surface
  group representations}.
\newblock http://arxiv.org/abs/0909.4487.

\bibitem[GP]{Abelian}
O.~{Garc\'ia-Prada} and A.~{Pe\'on-Nieto}.
\newblock {Higgs bundles, abelian gerbes and cameral data}.
\newblock In preparation.

\bibitem[GPR14]{HKR}
O.~{Garc\'ia-Prada}, A.~{Pe\'on-Nieto}, and S.~Ramanan.
\newblock {Higgs bundles and the Hitchin--Kostant--Rallis section}.
\newblock 2014.

\bibitem[GW09]{GW}
R.~Goodman and N.R. Wallach.
\newblock {\em {Symmetry, representations and invariants}}, volume 255 of {\em
  {Graduate Texts in Mathematics}}.
\newblock Springer-Verlag, 2009.
\newblock ISBN 978-0-387-79851-6.

\bibitem[Hit87a]{Duke}
N.~Hitchin.
\newblock {Stable bundles and integrable systems}.
\newblock {\em Duke Mathematical Journal}, 54(1):91--114, 1987.

\bibitem[Hit87b]{SDE}
N.J. Hitchin.
\newblock {The self duality equations on a Riemann surface}.
\newblock {\em Proc. London Math. Soc}, 55(3):59--126, 1987.

\bibitem[Hit92]{Teich}
N.J. Hitchin.
\newblock {Lie groups and {T}eichm{\"u}ller space}.
\newblock {\em Topology}, 31:449--473, 1992.

\bibitem[Hit13]{HitLanglands}
N.~Hitchin.
\newblock {Higgs bundles and characteristic classes}.
\newblock arXiv:1308.4603 [math.AG], 2013.

\bibitem[HP]{HP}
T.~Hausel and C.~Pauly.
\newblock {Prym varieties of spectral covers}.

\bibitem[HS14]{HitSchap}
N.~Hitchin and L.P. Schaposnik.
\newblock {Non-abelianization of Higgs bundles}.
\newblock {\em J. Differential Geom.}, 97(1):79--89, 2014.

\bibitem[HT03]{HT}
T.~Hausel and M.~Thaddeus.
\newblock {Mirror symmetry, Langlands duality and the Hitchin system}.
\newblock {\em Invent. Math.}, 153(1):197--229, 2003.

\bibitem[Kna05]{K}
A.W. Knapp.
\newblock {\em {Lie groups beyond an introduction}}, volume 140 of {\em
  {Progress in Mathematics}}.
\newblock Birkh{\"a}user, 2005.

\bibitem[Kos63]{Kos}
B.~Kostant.
\newblock {Lie Group Representations on Polynomial Rings}.
\newblock {\em American Journal of Mathematics}, 85(3):327--404, July 1963.

\bibitem[KP95]{KouPantevAut}
A.~Kouvidakis and T.~Pantev.
\newblock {The automorphism group of the moduli space of semisitable vecto
  bundles}.
\newblock {\em Math. Ann.}, 302(2):225--268, 1995.

\bibitem[KR71]{KR71}
B.~Kostant and S.~Rallis.
\newblock {Orbits and representations associated with symmetric spaces}.
\newblock {\em American Journal of Mathematics}, 93(3):753--809, July 1971.

\bibitem[Ng{\^o}10]{NgoLemme}
B.C. Ng{\^o}.
\newblock {Le lemme fondamental pour les alg\`ebres de Lie}.
\newblock {\em Publ. Math. Inst. Hautes Etudes Sci.}, (111):1--169, 2010.

\bibitem[PN13]{Tesis}
A.~Pe{\'o}n-Nieto.
\newblock {\em {Higgs bundles, real forms and the Hitchin fibration}}.
\newblock PhD thesis, Universidad Aut{\'o}noma de Madrid, 2013.

\bibitem[Sch13a]{LauThesis}
L.P. Schaposnik.
\newblock {\em {Spectral data for G-Higgs bundles.}}
\newblock PhD thesis, {Oxford University}, 2013.

\bibitem[Sch13b]{SchapUpp}
L.P. Schaposnik.
\newblock {Spectral data for U(m,m)-Higgs bundles}.
\newblock arXiv:1307.4419 [math.AG], 2013.

\bibitem[Sim97]{SimpsonHodge}
C.T. Simpson.
\newblock {The Hodge filtration on non-abelian cohomology}.
\newblock In {\em Algebraic Geometry--Santa Cruz, 1995. Proc. Sympos. Pure
  Math.}, volume~62, pages 217--281, 1997.

\end{thebibliography}
\end{document}